\documentclass[a4paper,12pt]{amsart}

\RequirePackage[OT1]{fontenc}
\RequirePackage{amsthm,amsmath}
\RequirePackage[colorlinks]{hyperref}

\numberwithin{equation}{section}

\usepackage[T1]{fontenc}
\usepackage[latin2]{inputenc}
\usepackage{indentfirst}
\usepackage[a4paper,left=2.5cm,right=2.5cm,top=2.5cm,bottom=2.5cm,bindingoffset=0cm]{geometry}
\usepackage{amsmath}
\usepackage{amssymb}
\usepackage{amsthm}
\usepackage{dsfont}
\usepackage{comment}
\usepackage{appendix}
\usepackage{graphicx}
\usepackage{subfigure}
\usepackage{enumerate}
\usepackage{pdfsync}

\newtheorem{thm}{Theorem}[section]    	
      	
\newtheorem{lemma}[thm]{Lemma}				
		
\newtheorem{defi}[thm]{Definition}
\newtheorem{conjecture}[thm]{Conjecture}

\newtheorem{example}[thm]{Example}
\newtheorem{assumption}[thm]{Assumption}

\newcommand*{\un}[1]{\underline{#1}}

\newcommand*{\be}{\begin{equation}}
\newcommand*{\ee}{\end{equation}}
\newcommand*{\ba}{\begin{aligned}}
\newcommand*{\ea}{\end{aligned}}
\newcommand*{\barr}{\begin{array}{c}}
\newcommand*{\earr}{\end{array}}
\newcommand*{\Ev}{{\mathbb{ E}}}
\newcommand*{\Pv}{{\mathbb{ P}}}

\newcommand*{\ind}{\mathds{1}}

\begin{document}






\title[A generalization of Barab\'asi priority model of human dynamics] {A generalization of Barab\'asi priority model of human dynamics}

\author{Júlia Komjáthy}
\address{Júlia Komjáthy, Department of Stochastics, Institute of Mathematics, Technical University of Budapest, 1521
Budapest, P.O.Box 91, Hungary} \email{komyju@math.bme.hu}

\author{K\'aroly Simon}
\address{K\'aroly Simon, Department of Stochastics, Institute of Mathematics, Technical University of Budapest, 1521
Budapest, P.O.Box 91, Hungary} \email{simonk@math.bme.hu}

\author{Lajos Vágó}
\address{Lajos Vágó, Department of Stochastics, Institute of Mathematics, Technical University of Budapest, 1521
Budapest, P.O.Box 91, Hungary} \email{vagolala@math.bme.hu}

 \thanks{
\indent This research was supported by the grant
 KTIA-OTKA  $\#$ CNK 77778, funded by the Hungarian National
Development Agency (NF\"U)  from a
source provided by KTIA.}


\begin{abstract}Albert-László Barabási introduced a model \cite{Bar} which exhibits the bursty nature of the arrival times of events in systems determined by decisions of some humans. In Barabási's model tasks are selected to execution according to some rules which depends on the priorities of the tasks. In this paper we generalize the selection rule of the A.-L. Barabási priority queuing model. We show that the bursty nature of human behavior can be explained by a model where tasks are selected proportional to their priorities. In addition, we extend some of Vázquez's heuristic arguments \cite{Vaz} to analytic proofs.

\end{abstract}
\date{\today}

\maketitle

\thispagestyle{empty}

\vspace{-0.7cm}

\maketitle

\section{Introduction}

Nowadays, many human-driven phenomena have become momentous in different areas of life, for example in economics and in social sciences. Classical models of human dynamics \cite{Rey, Gre} are based on Poisson processes, so in these models the elapsed time between two consecutive events, the so-called inter-event time is exponentially distributed.
    This is so because in these models there are many people who has effect on the system.
On the contrary, when a single person's actions are considered, in many real cases the inter-event time has a heavy-tailed distribution \cite{Bar}.
 That is the probability that the inter-event time is equal $k$ tends to zero with the speed of
 $k^{-\gamma } $ for a $\gamma >1$. (For precise definition see \cite[Ch. 1]{Remco})
 These heavy tailed distributions give relatively large weight to long inter-event times, and there are bursts of high activity between these inactive intervals. This is in sharp contrast with the exponential distribution, which gives low weight to the longer inter-event times.

Albert-László Barabási introduced a model \cite{Bar} which exhibited the bursty nature of the arrival times of events in systems determined by decisions of some humans. He analyzed it using simulations and found that it is in good agreement with empirical experiences. That is, for example the timing of e-mails sent by a user can be described with Barabási's model: The distribution of the inter-event time between two emails sent by a selected user can be approximated by power-law distributions with exponent close to one, which is also obtained by the model in the $p\to 1$ limit. Vázquez \cite{Vaz} was the first one who studied the model analytically. He confirmed Barabási's results by determining the exact distribution of the priorities and the waiting time at stationarity. He found that in the $p\to 1$ limit the distribution of the waiting time is close to a power-law distribution with exponent one.

In this paper, motivated by the work of Barabási \cite{Bar} and Vázquez \cite{Vaz}, we generalize this model to a more general one, where burstiness occurs as well. This generalization shows that this power-law decay with exponent $1$ can be obtained in sequences of models where tasks with higher priority are more and more likely to be chosen. In addition, we extend some of Vázquez's heuristic arguments to analytic proofs.

\textbf{Barabási's model} is as follows: Somebody has a todo list which always consists of exactly  $L$ items, each of which has a non-negative priority. In every discrete time step ($1,2,\dots$) two things happen:
\begin{itemize}
  \item A task is selected for execution according to a selection protocol described later and leaves the system,
  \item another task arrives to the list.
\end{itemize}
The priority of all arriving task including the tasks at the  $0$-th step are i.i.d. random variables with an absolutely continuous distribution function $R(x)$ on $[0, \infty)$. Let us call a task \emph{new} at time $t$ if it has just arrived to the list. We denote by $N_t$ the priority of the new task and let us introduce
\[
R(x):=\Pv(N_t\le x).
\]
The corresponding density function (DF) is
\[
r(x):=R'(x).
\]
Throughout the paper we assume that
\[
r\in L^2(0,1).
\]
In addition, let us call a task old at time $t$, if it was on the list at time $t-1$ at it is still there at time $t$. We denote its priority by $O_t$ and denote its cumulative distribution function by
\[
R_1(x,t):=\Pv(O_t \le x),
\]
and its density
\[
r_{1}(x,t):=\frac{d}{dx}R_{1}(x,t).
\]

 The selection protocol in  Barabási's model is as follows: Independently in each step independently, we toss a coin with probability of heads $0\leq p\leq 1$. If heads, the task with the highest priority is selected for execution and leaves the system. If tails, a task selected uniformly from the list leaves the system. In particular, if $p=1$, then we use the highest priority first selection protocol, and if $p=0$, then the selection is completely uniform. Let us denote by $\tau$ the waiting time of a task to for execution, that is, the number of steps between its arrival at and its departure from the system. Using computer simulations, Barabási found
 that if $p\to 1$ then
 the waiting time $\tau$  has a power law tail with exponent $1$ . These asymptotics  hold not only for long lists but even if the list consists of two items, that is  for $L=2$. In this case Vázquez \cite{Vaz} proved analytically that
\begin{equation}\label{502}
   \lim_{p\to 1}\Pv (\tau = k)=\begin{cases} 1+O\left( \frac{1-p}{2}\ln(1-p)\right) &\text{ if $k=1$,}\\
  O\left( \frac{1-p}{2}\right)\frac{1}{k-1} &\text{ if $k>1$.}\end{cases}
\end{equation}

In \cite{Vaz}  the expected value of the waiting time was also investigated.
 For the case $L=2$ Vázquez proved analytically  that
 $$
 \Ev (\tau)=
 \left\{
   \begin{array}{ll}
     2, & \hbox{if p<1;} \\
     1, & \hbox{p=1}
   \end{array}
 \right.
 $$
in stationarity. If $p=1$ this means that there is a task stuck in the queue with priority $0$.
For the case $L>2$  Vázquez conjectured that
 \begin{conjecture}[Vázquez]\label{501}
 $$
 \Ev (\tau)=
 \left\{
   \begin{array}{ll}
     L, & \hbox{if p<1;} \\
     1, & \hbox{p=1}
   \end{array}
 \right.
 $$
in  stationarity. If $p=1$  stationarity means that there are $L-1$ tasks stuck.
 \end{conjecture}

Our work is divided into two main parts.
\subsection{Barabási model}
We extend some of the  heuristic arguments of \cite{Vaz}  related to the Barabási model to analytic proofs in Section \ref{303}. We distinguish two cases according to $p<1$ or $p=1$.
\subsubsection{Barabási model with $p<1$}
First we assume that $0\leq p<1$. We extend Vázquez's completely correct heuristic argument to an analytic proof in Section \ref{3031}.
\subsubsection{Barabási model with $p=1$}
Then in Section \ref{800} we consider the case when at every time step the task with the highest priority is selected. Although the relevant part of Vázquez's Conjecture \ref{501} holds but the stationary case will not be achieved, since there will never be any tasks with priority $0$ almost surely. It is of crucial importance to observe that in this case at time $t$ the $L$ items on the list are as follows: Beside the newly added item we have the  $L-1$ lowest priority items among all items which have ever added to the list.  Therefore we can use the theory of the process of records. We study this case for $L=2$  in Section \ref{800}, where all tasks remain on the list for one time step except for the records. An item is called lower record (from which we usually omit to say the adjective lower) if it has the lowest priority at the time when it is added to the list. When such a record arrives, then it remains on the list until a new record arrives. The time elapsed between the $n$-th and $n+1$-th records called the $n$-th inter-record time. It is well known that the expected value of the inter-record times grow exponentially with $n$. Therefore simulations may indicate the false conclusion that one task remains indefinitely in the queue \cite{Vaz} \cite{Bar}.
\subsection{Generalization of the Barabási model}
In Section \ref{420} We study the Barabási-model for general priority-based selection protocols in the case $L=2$. We analyze the two main characteristics of this model at stationarity, namely the distribution of the priorities of tasks on the list and that of $\tau$, the waiting time a task spends on the list before execution. To explain these with more detail, we need to introduce some new notation. At time $t+1$ there are two items on the list. The one which has just arrived to the list we call \emph{new}, and the one which was already on the list at time $t$ we call it \emph{old} task. The distribution of the new task is given by $R(x)$. Our goal is to understand the distribution of the old task. The selection protocol is described by the function $v$ in the following way:
\[
  v(x,y):=\Pv (\text{new is chosen} \mid \text{new has priority $x$, old has priority $y$}),
\]
where $\forall y\ v(.,y)$ is increasing and $\forall x\ v(x,.)$ is decreasing. Furthermore we assume that $\forall x\ \forall y\ v(x,y)\leq c_1 < 1$ with some constant $c_1$.

We use Vázquez's notation, and also refer to the results in \cite{Vaz} for the \emph{Barabási case}, when $v(x,y)=p\ind_{x>y}+\frac{1-p}{2}$ (the model introduced by Barabási \cite{Bar}). An other natural example is as follows.
\begin{example}\label{321}
Let $v$ be $v(x,y)=p\frac{x}{x+y}+\frac{1-p}{2}$, $0\leq p\leq 1$ and the priority of the new task is uniformly distributed on $[c,1]$, $0<c<1$. That is, first we toss a biased coin (heads with probability $p$). If heads, we select a task for execution proportionally to the priorities, if tails, the selection is done uniformly. This model takes into account not only the order of priorities, but also the proportion of those.
\end{example}

If tasks are interpreted as being competitors and the priorities as levels of talent of these competitors, then Example \ref{321} can be explained as follows: In a game, every competitor entering the game have to play until he wins for the first time. We also know that the outcome of the games depends on the level of talents as it is determined by the selection protocol of Example \ref{321}. That is, with probability $p$ the chances to win are proportional to the level of talents of the two competitors, and with probability $1-p$ the competitors win with $\frac12$-$\frac12$ probability. The a priori distribution of the talent of a player is uniform on $[c,1]$, and the game has been going on for a long time. We are interested in the distribution of the level of talent of a player about whom we only know that he has just lost a game. Then the corresponding density function $r_1(x)$ is illustrated in Section \ref{4002}, see Fig. \ref{410}.

Very useful notations are the following ones. The probability that the new task is selected given the old task has priority $s$ is
\be\label{507}
  q(s)=\int_{0}^{1}{v(y,s)dR(y)},
\ee
and the probability that the new task is selected given the old task has priority $s$ is
\[
  q_1(s,t)=\int_{0}^{1}{(1-v(s,y))dR_1(y,t)}.
\]
With these notations one can easily describe the evolution of $R(x,t)$. Assume that at time $t+1$ the old task is $T^*$. Then either
\begin{enumerate}
	\item[(a)] $T^*$ was the old task at time $t$, or
	\item[(b)] $T^*$ was the new task at time $t$.
\end{enumerate}
In the first case (a), the new task had to be selected for execution in step $t$, so that $T^*$ remained in the system, and in the other case (b) the old task had to be selected so that $T^*$ remained in the system. Hence law of total probability gives us
\[
  R_{1}(x,t+1)=\underbrace{\int_{0}^{x}{r_{1}(s,t)q(s)ds}}_{case(a)} +\underbrace{\int_{0}^{x}{r(s)q_{1}(s,t)ds}}_{case(b)}.
\]
Assuming the system is in stationary state, the probability that the old task is selected given the new task has priority $s$ is
\be\label{5071}
  q_1(s)=\int_{0}^{1}{(1-v(s,y))dR_1(y)},
\ee
and for the stationary CDF $R_1(x)$ and DF $r_1(x)$ of the priority of the old task we obtain the stationary equation
\be\label{508}
  R_{1}(x)=\int_{0}^{x}{r_{1}(s)q(s)ds}+\int_{0}^{x}{r(s)q_{1}(s)ds}.
\ee
In the very special \emph{Barabási case} $R_1(x)$ can be explicitly computed \cite{Vaz}. However, in the general case the function $R_1(x)$ cannot be expressed by a closed formula. We will use Hilbert\--\ Schmidt operator techniques to find an approximation of $R_1(x)$.

As far as the waiting time is concerned, the distribution of $\tau$ is:
\be \label{310}
  \mathbb{P}(\tau=k)=
  \begin{cases}
    \int_{0}^{1}{(1-q_{1}(x))dR(x)} & \text{if $k=1$,}\\
    \int_{0}^{1}{q_{1}(x)(1-q(x))q(x)^{k-2}dR(x)} & \text{if $k>1$.}
  \end{cases}
\ee
If $k=1$, then the task has to be executed immediately, so the probability of this is exactly $1-q_{1}(x)$ if the task has priority $x$. Otherwise, if $k>1$, then the event that a task with priority $x$ stays exactly $k$ steps on the waiting list is the intersection of $k$ conditionally independent events. When the task is added to the list, it must not be executed, this happens with probability $q_{1}(x)$, then it becomes old and has to stay $k-2$ more steps on the list, the probability of this event is $q(x)^{k-2}$, and finally, the task has to be executed, which event has probability $1-q(x)$.

In the \emph{Barabási case} these integrals can be explicitly computed \cite{Vaz}, hence we obtain
\[
  \mathbb{P}(\tau=k)=
  \begin{cases}
    1-\frac{1-p^{2}}{4p}ln\frac{1+p}{1-p} & \text{if $k=1$,}\\
    \frac{1-p^{2}}{4p}\left(\left(\frac{1+p}{2}\right)^{k-1}-\left(\frac{1-p}{2}\right)^{k-1}\right)\frac{1}{k-1} & \text{if $k>1$.}
  \end{cases}
\]
In the limit $p\to 1$ it leads to \eqref{502}.

In the general case only the expected waiting time can be explicitly computed, see Section \ref{3031}.

\section{Probabilistic interpretation of $R_1(x)$ in the Barabási case}\label{306}

Consider the Barabási case with $p<1$. As we mentioned earlier, Vázquez  \cite{Vaz}
computed $R_1(x)$ in the Barabási case.
In order to give a better understanding of the nature of the model,
 here we provide an alternative, probabilistic way to compute the CDF $R_1(x)$.
Part of this  method will be used when we prove Conjecture \ref{501}.

\medskip

To get the distribution function $R_1(x)$ of the old task in a stationary system, first extend the dynamics backwards in time such that the values of the system at $-1,-2, \dots$ are defined to maintain stationarity. Further, let us suppose that the process is stationary at time $T$.
 Starting from this $T$, we are looking backwards in time and count how many other tasks had an old task to compete with in order to be able to stay in the list until $T$. We define the following three events:
\begin{itemize}
  \item
  $
  R_{\mathrm{new}} :=\big\{
  \mbox{The selection protocol turns out to be the random choice and}\\
  \mbox{ a new task is executed.}
  \big\}
  $\\
In this case the new task leaves the list immediately. Thus, the old task does not have to compete with its priority at all, so we do not have to count these cases. This event happens in each step with probability $\frac{1-p}{2}$.
  \item
  $
  R_{\mathrm{old}}:=\big\{
  \mbox{Random protocol is chosen, and the old task is selected.}
  \big\}
  $\\
In this case the arriving new task becomes old and the old task before it leaves the system. Since the priority distribution of the new task is just $R(x)$, considering the old task's priority distribution, the system restarts at each event of this type. This event happens in each step with probability $\frac{1-p}{2}$.
  \item
  $
  R_{\mathrm{comp}}:=\big\{
  \mbox{Priority selection.}
  \big\}
  $\\
This third event happens with probability $p$.
\end{itemize}
Thus the process can be coded by sequences of type $\{R_{\mathrm{new}}, R_{\mathrm{old}}, R_{\mathrm{comp}}\}^{\mathbb Z}$.

Let $X_t$ be the $t$-th element of this sequence. Then the renewal at every $R_{\mathrm{old}}$ event can be formalized as follows: Using the fact that the sequence $\{ X_t\}_{t=1}^{\infty}$ is independent of the priorities, for every $x\in \mathbb{R}$
\[
\Pv(O_{t+1} < x \mid X_t=R_{\mathrm{old}}) = \Pv( N_t < x \mid X_t=R_{\mathrm{old}}) = \Pv( N_t < x)=R(x),
\]
where recall that $O_t$ and $N_t$ stands for the priority of the \emph{old} and the \emph{new} task at time step $t$, respectively.

Since the priority distribution of the old task is renewed at every event $R_{\mathrm{old}}$, in order to determine the priority distribution of the old task in the system for a big enough $T$  , we have to count how many $R_{\mathrm{comp}}$-s we had after the last $R_{\mathrm{old}}$ event.

It is easy to see that the distribution of the arrival time of the last $R_{\mathrm{old}}$ event is just $T-GEO\left( \frac{1-p}{2}\right)+1$. Similarly, if we only want to count the number of $R_{\mathrm{comp}}$-s after the last $R_{\mathrm{old}}$, we have to re-normalize the probabilities to exclude cases of $R_{\mathrm{new}}$, so it's distribution is
\[
GEO\left( \frac{\frac{1-p}{2}}{p+\frac{1-p}{2}}\right) -1= GEO\left( \frac{1-p}{1+p}\right) -1.
\]

Thus, the old task in the system at time $T$ had to compete with the other tasks at each event $R_{\mathrm{comp}}$ and had to "win", i.e its priority had to be less than all of them, in order to stay in the system up to time $T$. So, its priority had to "defeat" $GEO\left( \frac{1-p}{1+p}\right)-1$ many other tasks to stay in the list so far. Hence the distribution we were looking for is the minimum of $X\sim GEO\left( \frac{1-p}{1+p}\right)$ many independent random variables ($\{ Y_{i}\} _{i=1}^{X}$), of distribution $R(x)$. Let us compute it's CDF using the law of total probability:
\[ \ba
F(x)&=\mathbb{P}(\min_{i=1 \dots X} Y_i < x)= \sum_{k=1}^{\infty}{\mathbb{P}(\exists i \le X: Y_{i}\leq x \mid X=k)\mathbb{P}(X=k)}\\
&=\frac{1+p}{2}\left( 1-\frac{1}{1+\frac{2p}{1-p}R(x)}\right),
\ea \]

in agreement with \cite{Vaz}, since the summation converges if $p<1$. A comment is that we only used the extension to doubly infinite stationary sequence since the geometric variable in this proof is unbounded.

\section{Records and expected waiting time}\label{303}
\subsection{Ergodicity for $p<1$}\label{3031}
In this section we verify Conjecture \ref{501} which gives us the expected waiting time in the Barabási case. Let us first set $p<1$ and $L\geq 2$. We use the method suggested by Vázquez \cite{Vaz}, and work out the details. Using the arguments introduced by Vázquez let us denote by $\tau_{i}$ the waiting time of the task executed at step $i$, and by $\tau_{l,t}'$ ($l=1,\dots,L-1$) the resident times of the tasks that are still in the buffer at step $t$. Then we summarize the main point in the following theorem:

\begin{thm}\emph{[Ergodic Theorem for $L\ge 2, p<1$.]}\label{thm::ergod}
The priority queuing system in the Barabási case is ergodic for any $L\geq2$ and $p<1$ and the following limit holds:
\be \label{311}
  \Ev_L (\tau) = \lim_{t\to \infty }\frac1t \sum_{i=1}^{t}{\tau_{i}} = L \quad a.s. \text{ and in } L_1.
\ee
\end{thm}

\begin{proof}
Since the total buffer time until time $t$ is $Lt$ and in each time step exactly one task leaves the system, we have that
\[
  \sum_{i=1}^{t}{\tau_{i}}+\sum_{l=1}^{L-1}{\tau_{l,t}'}=Lt.
\]
Hence
\be \label{308}
  \lim_{t\rightarrow \infty}\frac{1}{t}\sum_{i=1}^{t}{\tau_{i}} =L-\lim_{t\rightarrow\infty}\frac{1}{t}\sum_{l=1}^{L-1}{\tau_{l,t}'}
\ee
if these limits exist. Similarly to the $L=2$ case in Section \ref{306}, the process renews whenever there are $L-1$ successive $R_{\mathrm{old}}$ events, where $R_{\mathrm{old}}$ means that the oldest task is selected for execution because of the random selection. For every $t$ the probability that there is a renewal section of the process starting at the step $t$ is $\left(\frac{1-p}{L}\right)^{L-1}$. In addition, two renewal sections starting at time $t_1$ and $t_2$ are independent if $|t_1-t_2|\geq L-1$. So the system renews infinitely often a.s.. (This implies also the existence of a stationary distribution for the $L-1$ not newly added task's priority.)
Analogous to the argument in Section \ref{306}, each remaining time $\tau_{l,t}'$ can be stochastically dominated by a geometric random variable with success probability $(\frac{1-p}{L})^{L-1}$. Thus,  $\frac{\tau_{l,t}'}{t} \to 0$ as $t\to \infty$ holds for any $l$ almost surely (a.s.) and in $L_1$, and so we obtain
\be\label{312}
  \lim_{t\rightarrow\infty}\frac{1}{t}\sum_{l=1}^{L-1}{\tau_{1,t}'}=0 \text{ a.s. and in $L_1$}.
\ee

There remains the proof of ergodicity. Let us denote by $\un M_t= (M^1_t, \dots, M^{L-1}_t)$ the $L-1$ dimensional vector of the oldest elements in the buffer at time $t$, such that $M^1_t$ is the priority of the oldest, $M^i_t$ is the priority of the $i$-th oldest task in the system. Then clearly the sequence $\{ \un{M}_t\}_{t=1}^{\infty}$ is Markovian with state space $[0,1]^{L-1}$. Moreover, it is clearly aperiodic, and since the system renews infinitely often a.s., it is irreducible. Thus ergodicity and equations \eqref{308} and \eqref{312} immediately imply \eqref{311} as it was suggested by Vázquez's heuristic arguments \cite{Vaz}.
\end{proof}

\subsection{Records}\label{800}

In this section we investigate the case $p=1, L=2$ , i.e. the case when the selection protocol has no randomness and always picks the task in the system with higher priority. We show that the system achieves no stationary distribution in this case.

It is easy to see that when a task arrives whose priority is less than all the priorities before, then the old task in the system will be chosen for execution and this task will remain in the buffer. This implies that $R_{1}(x,t)$ (the distribution of the old task's priority at time $t$) is the distribution of the minimum, i.e. the lower record value of the first $t+2$ tasks' priority.  That is, we need to consider the minimum of $t+2$ i.i.d. random variables with distribution $R(x)$.

On the other hand, a stationary distribution $R_1^s(x)$ should satisfy equation \eqref{508}, which has only a degenerate solution $R_{1}(x)=0$ a.s. if $p=1$.
Thus, the stationary distribution $R_1^s(t)$ will never be achieved, since
\[
\mathbb{P}(\forall t\ R_{1}(x,t)\neq 0)=1.
\]
Hence, the properties of the waiting time $\tau$ are related to the process of records.

In order to introduce some classical theorems of record theory (for an introduction see \cite[p. 22-28]{ABN}), we will need some definitions.
\begin{defi} Let $\{ X_{t}\}_{t=1}^{\infty}$ be i.i.d. random variables with absolutely continuous CDF $F$. Then we define the followings:
\begin{enumerate}
	\item The sequence $\{ T_{k}\}_{k=1}^{\infty}$ of \emph{lower record times} is defined by
	\[\begin{aligned}
	T_{1}&=1, \text{and for $k>1$}\\
	T_{k}&=\min\{ t \mid X_{t}<X_{T_{k-1}}\}.
	\end{aligned}\]
	\item The sequence $\{ \Delta_{k}\}_{k=2}^{\infty}$ of \emph{inter-record times} is defined by
	\[
	\Delta_{k}=T_k-T_{k-1}.
	\]
	\item The sequence $\{ I_{t}\}_{t=1}^{\infty}$ of \emph{record indicators} is defined by
	\[\begin{aligned}
	I_{1}&=1, \text{and for $t>1$}\\
	I_{t}&=\mathds{1}\{X_t<\min\{ X_1,\dots,X_{t-1}\}\}.
	\end{aligned}\]
	\item The sequence $\{ x_{k}\}_{k=1}^{\infty}$ of \emph{lower record values} is defined by
	\[
	x_{k}=X_{T_{k}}.
	\]
\end{enumerate}
\end{defi}

If we do not care about the record values, without loss of generality one can assume that the distribution of the $X_i$-s is $U[0,1]$. Thus one can easily compute the joint distribution of the $I_t$-s. Let $n,1=t_1<t_2<\dots <t_n$ be arbitrary positive integers. Then
\[ \begin{aligned}
\Pv(I_{t_1}&=1,\dots,I_{t_n}=1)\\
&=\int_{0<x_1<\dots<x_n<1}{\Pv(I_{t_1}=1,\dots,I_{t_n}=1 \mid X_{t_1=x_1},\dots,X_{t_n=x_n})}dx_1\dots dx_n\\
&=\int_{0<x_1<\dots<x_n<1}{x_1^{t_2-1}x_2^{t_3-t_2-1}\dots x_n^{t_n-t_{n-1}-1}dx_1\dots dx_n}\\
&=\frac{1}{t_1}\frac{1}{t_2}\dots \frac{1}{t_n}.
\end{aligned} \]
Since the integers $n,1=t_1<t_2<\dots <t_n$ were arbitrary, and the $I_t$-s are discrete random variables, it is evident that these random variables are independent, and $I_t\sim BER\left(\frac1t\right)$ is a Bernoulli random variable.

Now consider the ratio $\frac{\Delta_k}{T_k}$. The distribution of $\frac{\Delta_k}{T_k}$ conditioned on the value of $T_{k-1}$ can be determined by Tata's reasoning \cite{Tata}:
\[ \begin{aligned}
\Pv\left(\frac{\Delta_k}{T_k}>x \;\Big\vert\; T_{k-1}=t\right) &=\Pv\left(\frac{\Delta_k}{t+\Delta_k}>x \;\Big\vert\; T_{k-1}=t\right) \\
&=\Pv\left( \Delta_k>\frac{x}{1-x}t \;\Big\vert\; T_{k-1}=t\right)\\
&=\Pv\left(I_{t+1}=0,\dots,I_{	\lfloor t/(1-x)	 \rfloor}=0\right)\\
&=\frac{t}{\lfloor t/(1-x)\rfloor}.
\end{aligned} \]
Since $T_k \to \infty$ a.s. as $k \to \infty$, hence
\[
\lim_{k \to \infty} \Pv\left(\frac{\Delta_k}{T_k}>x\right)=1-x,
\]
so thus $\frac{\Delta_k}{T_k}$ and $\frac{T_{k-1}}{T_k}$ are asymptotically $U[0,1]$.

Now we turn back to the Barabási model in the case $p=1$ and $L=2$. Equation \eqref{308} holds also in this situation:
\[
  \lim_{t\rightarrow \infty}\frac{1}{t}\sum_{i=1}^{t}{\tau_{i}}=2-\lim_{t\rightarrow \infty}\frac{\tau_{1,t}'}{t}
\]
if these limits exist. But now, with the help of record theory one can prove that the limit on the right hand side does not exist. On the one hand if we choose $t = T_k +1$ then we obtain:
\[
  \liminf_{t\rightarrow \infty}\frac{\tau_{1,t}'}{t}=0 \text{\ a.s.,}
\]
since the sequence of records is infinite a.s. (This is ensured by the continuity of the CDF $F$), so thus $\tau_{1,t}'=0$ will hold for infinitely many $t$-s. On the other hand the $limsup$ is taken if $t=T_k$ (we look the subsequence of records):
\[ \begin{aligned}
\limsup_{t\rightarrow \infty}\frac{\tau_{1,t}'}{t}&=\lim_{k\to \infty}\frac{\Delta_{k-1}}{T_k}=\lim_{k\to \infty}\frac{T_k-T_{k-1}}{T_k}\\
&=1-\lim_{k\to \infty}\frac{T_{k-1}}{T_k}=1-Z,
\end{aligned} \]
where $T_k$ and $\Delta_k$ are the $k$-th record and inter-record times, and $Z$ is a $U[0,1]$ random variable.

Thus, Vázquez's heuristics does not work in this case. Note that the dynamics is not even ergodic, since the irreducibility does not hold, as the priority of the old task decreases monotonically.

Further, we can state precise asymptotic results about the waiting time of the records in the queue.
It was proven by  Holmes and Strawderman \cite{HS} that the Strong Law of Large Numbers (SLLN), by Neuts \cite{Neu} that the Central Limit Theorem (CLT) and by Strawderman and Holmes \cite{SH} that the Law of Iterated Logarithm (LIL) hold for the series of $\ln T_{k}$ and $\ln \Delta_{k}$ as well:
\begin{thm}
Let $z_{k}$ be either $\Delta_{k}$ or $T_{k}$. Then
\begin{enumerate} \item \emph{[SLLN]\cite{HS}}
\[
\frac{\ln z_{k}}{k}\stackrel{a.s.}{\rightarrow} 1 \text{\ as $k\rightarrow \infty$}.
\]
\item \emph{[CLT]\cite{Neu}}
\[
\frac{\ln z_{k}-k}{\sqrt{k}}\Rightarrow N(0,1) \text{\ as $k\rightarrow \infty$}.
\]
\item \emph{[LIL]\cite{SH}}
\[\begin{aligned}
\limsup_{k\rightarrow \infty} \left\{ \frac{\ln z_{k}-k}{\sqrt{2k\ln\ln k}}\right\} &=1 \text{\ a.s.},\\
\liminf_{k\rightarrow \infty} \left\{ \frac{\ln z_{k}-k}{\sqrt{2k\ln\ln k}}\right\} &=-1 \text{\ a.s.\ }.
\end{aligned}\]
\end{enumerate}
\end{thm}
Thus we see that the $k$-th record spends roughly $e^k$ time in the buffer.


\section{A natural generalization of the model}\label{420}

In this section we consider a common generalization of Barabási's model and Example \ref{321} which were not analyzed in the literature before. Also in this new model there is a list with $L$ tasks, the new tasks has i.i.d. non-negative priorities, and in each discrete time step one task is selected to execution and a new task arrives to it's place. The only change is that the dynamics of the system is now given by a general selection protocol. We will analyze this new model for the case $L=2$, so we define the selection protocol to this case. In every time step we select a task according to the priorities in the following way:
\be \label{504}
v(x,y):=\Pv (\text{new is chosen} \mid \text{new has priority $x$, old has priority $y$}),
\ee
where
\begin{assumption}\label{assump}
We assume that $\forall y\ v(.,y)$ is increasing and $\forall x\ v(x,.)$ is decreasing. Furthermore let us suppose that
\[ \forall x\ \forall y\ v(x,y)\leq c_1 < 1\] with some constant $c_1$.
\end{assumption}

Let  $O_t$ be the priorities of the old task  at time $t$. Clearly, the sequence $\left\{O_t\right\}_{t=1}^{\infty }$ is an aperiodic Markov-chain on the state space $(0,1)$.

\begin{lemma}\label{505}
Under Assumption \ref{assump}, the Markov-chain $\{ O_t\}_{t=1}^{\infty}$ defined above is positive recurrent.
\end{lemma}
First note that positive recurrence of the Markov chain implies the existence of a stationary distribution as well.
We remind the reader that $q(s)$ and $q_1(s)$ was defined in \eqref{507} and \eqref{5071}, and $r(x)$ is the density function of the new task. Since $\exists c_1$ $\forall x\ \forall y\ v(x,y)\leq c_1 < 1$, hence $\forall s\ q(s)=\int_{0}^{1}{v(y,s)dR(y)}\leq c_1<1$ and $\forall s\ q_1(s)\geq 1-c_1>0$.
\begin{proof}
For any $x\in supp(r)$ let us denote by $T_{x,y,\epsilon}$ the hitting time from $x$ to $(y-\epsilon,y+\epsilon)$. Thus the distribution of $T_{x,y,\epsilon}$ is stochastically dominated by the geometric distribution with success probability $(1-c_1)(R(y+\epsilon)-R(y-\epsilon))$. Therefore
\[
  \Ev(T_{x,y,\epsilon})\leq \frac{1}{(1-c_1)(R(y+\epsilon)-R(y-\epsilon))}<\infty.
\]
\end{proof}
Since the Markov chain is ergodic, this fact implies that Theorem \ref{thm::ergod} is valid in this case, and the expected waiting time of a task equals $2$.

\subsection{Distribution of the priorities}
In the following we write the density function $r_1(x)$ of the old task's priority at stationarity as a fix point of a Hilbert-Schmidt operator. Then using this fix point equation we show an opportunity of approximating $r_1(x)$, and then we approximate $r_1(x)$ of Example \ref{321}.

Let us consider the general case. We know from Lemma \ref{505} that stationarity distribution exists for the chain. The stationary equation \eqref{508} is
\[
  R_{1}(x)=\int_{0}^{x}{r_{1}(s)q(s)ds}+\int_{0}^{x}{r(s)q_{1}(s)ds}.
\]
After substituting \eqref{507} and \eqref{5071} we obtain
\[
  R_{1}(x)=\int_{0}^{x}{\int_{0}^{1}{(r_1(s)r(y)v(y,s)+r(s)r_1(y)(1-v(s,y)))dyds}}
\]
Differentiation with respect to variable $x$ yields the following equation for the stationary density function $r_1(x)$ of the old task's priority:
\[
  r_{1}(x)=\int_{y=0}^{1}{(r_1(x)r(y)v(y,x)+r(x)r_1(y)(1-v(x,y)))dy}
\]
With the constant $c_1$ from the definition of $v$ \eqref{504} we can write:
\[
  r_{1}(x)\left(1-\int_{y=0}^{1}{r(y)v(y,x)dy}\right)=r(x)\left( \int_{y=0}^{1}{r_1(y)(1-v(x,y))dy}-c_1\right)+c_1\:r(x).
\]
It is easy to see that the coefficient of $r_1(x)$ on the left side cannot be zero, hence
\be \ba\label{510}
  r_{1}(x)&=\int_{y=0}^{1}{\frac{r(x)(1-v(x,y)-c_1)}{1-\int_{z=0}^{1}{r(z)v(z,x)dz}}r_1(y)dy}\\
  &+\frac{c_1\:r(x)}{1-\int_{z=0}^{1}{r(z)v(z,x)dz}}.
\ea \ee

To shorten the notation we introduce
\be\label{511}
\ba
  &\alpha(x,y):=1-v(x,y)-c_1,\\
  &g(x):=\frac{r(x)}{1-\int_{z=0}^{1}{r(z)v(z,x)dz}},\\
  &K(x,y):=\alpha(x,y)g(x),\text{\ \ \ and}\\
  &f(x):=\frac{c_1\:r(x)}{1-\int_{z=0}^{1}{r(z)v(z,x)dz}}.
\ea\ee
Moreover we introduce the integral operator $A$ associated to kernel function $K$, i.e.
\[
  A\varphi(x):=\int_{y=0}^{1}{K(x,y)\varphi(y)dy}.
\]
Since $K\in L^2[0,1]\times [0,1]$, hence $A$ maps $L^2[0,1]$ into itself (see \cite[Theorem 9.2.1.]{KF}). With this notation, equation \eqref{510} can be written as
\[
(I-A)  r_{1}(x)=f(x).
\]

Hence, assuming that the operator $I-A$ is invertible for certain choice of the underlying distribution $R$ and the selection protocol $v$, we would like to determine the function $r_1(x)$ which satisfies
\be\label{512}
  r_1(x)=(I-A)^{-1}f(x)
\ee
If $\sum_{n=0}^{\infty}{A^{n}f(x)}$ converges we could write
\[
  r_1(x)=\sum_{n=0}^{\infty}{A^{n}f(x)}.
\]
It turns out that $A^n$ is also an integral operator with some kernel function $K_n$, which is
\[
  K_{n}(x,y)=\underbrace{\int_{0}^{1}\dots\int_{0}^{1}}_{n-1}{{K(x,u_{1})K(u_{1},u_{2})\dots K(u_{n-1},y)du_{1}\dots du_{n-1}}}.
\]
Note that we can benefit from the fact that $f(x)=c_1 g(x)$ (see equation \eqref{511}), so we get
\[\ba
  A^{n}f(x)&=A^{n}g(x)c_1=c_1\int_{0}^{1}{K_{n}(x,y)g(y)dy}\\
&=c_1\int_{0}^{1}\dots\int_{0}^{1}{{\alpha(x,u_{1})g(x)\dots \alpha(u_{n-1},t)g(u_{n-1})g(y)du_{1}\dots du_{n-1}dy}}\\
&=c_1\:g(x)\underbrace{\int_{0}^{1}\dots\int_{0}^{1}{\underbrace{\alpha(x,u_{1})g(u_{1})}_{\widetilde{K}(x,u_{1})}\dots \alpha(u_{n-1},u_{n})g(u_{n})du_{1}\dots du_{n}}}_{H_{n}(x)}.
\ea\]
Thus we see that it is useful for us to introduce another kernel function given by
\be \label{def::kernel_k_tilde}
\widetilde{K}(x,y):=\alpha(x,y)g(y),
\ee
and the corresponding integral operator
\be \label{def::operator_a_tilde}
\widetilde{A}h(x):=\int_{0}^{1}{\widetilde{K}(x,y)h(y)dy},
\ee
which maps $L^2[0,1]$ into itself. With this notations, the terms $A^n f(x)$ can be obtained as $c_1g(x) H_n(x)$, where
\[\ba
H_{n}(x)&=\int_{0}^{1}\dots\int_{0}^{1}{\alpha(x,u_{1})g(u_{1})\dots \alpha(u_{n-1},u_{n})g(u_{n})du_{1}\dots du_{n}}\\
&=\int_{0}^{1}\dots\int_{0}^{1}{\widetilde{K}(x,u_{1})\dots \widetilde{K}(u_{n-1},u_{n})du_{1}\dots du_{n}}=\widetilde{A}^{n}\mathds{1}(x).
\ea\]
So we see that \eqref{512} is equivalent to
\be\label{408}
r_1(x)=c_1\:g(x)(1+H_{1}(x)+H_{2}(x)+\dots),
\ee
which yields:
\[
r_1(x)\approx c_1\:g(x)(1+H_{1}(x)+H_{2}(x)+\dots +H_n(x)).
\]

\subsection{Example}\label{4002}

In the following we consider Example \ref{321}. Recall that in this case, priority service happens with probability $p$, and in this case a task is chosen proportional to its priority, i.e. $v(x,y)= p x/(x+y) + (1-p)/2$. In the sequel, using the software \emph{Wolfram Mathematica}, we determine $r_{1}(x)$ numerically for this particular choice with some fixed parameters. We compute the terms in \eqref{408} and  give estimates on $H_n(x)$. From these we can derive estimates on parameters $(p,c)$ where the series in \eqref{408} converges. Before doing these, we argue shortly that shifting the distribution of the new task's priority up by $c$ has a real relevance in the literature. To see this, note that if $X$, $Y\sim U[c,1]$, and $V$, $W\sim U[0,1]$, then
\[
\frac{X}{X+Y} \sim \frac{V(1-c)+c}{V(1-c)+c+W(1-c)+c} \sim \frac{V+\delta}{V+W+2\delta},
\]
where $\delta=\frac{c}{1-c}\geq 0$. This protocol is similar to the growth rule of the preferential attachment model, see \cite{BA,BRST,BR}.

Recall the definition of the kernel function $\widetilde K$ and the corresponding operator $\widetilde A$. We have seen that $H_n(x)=\widetilde{A}^{n}\ind(x)$. So in order to prove convergence of \eqref{408} it is enough to give an upper bound on the $L^2$ norm of $H_n(x)$. This can be achieved by estimating the $L^2$ norm of $\widetilde{A}$ or the Hilbert-Schmidt norm of $\widetilde K$:
\[
\| H_n(x)\|_2\leq \| \widetilde{A}\|_2^n\leq \| \widetilde{K}\|_{HS}^n,
\]
where $\| \widetilde{K}\|_{HS}^2=\int_c^1{\int_c^1{\widetilde{K}^2(x,y)dy}dx}$.
If $\| \widetilde{K}\|_{HS}^2<1$ then the sum in equation \eqref{408} converges, so it remains to find the parameter region $(p,c)$ for which this holds.
Using the definitions given in Example \ref{321} and in equations \eqref{511} and \eqref{def::kernel_k_tilde} we get that
\[ \begin{aligned}
\| \widetilde{K}\|_{HS}^2&=\int_c^1{\int_c^1{\left( \frac{y}{y+x}\:\frac{1}{\frac{1+p}{2p}-\left(1-\frac{x}{1-c}\ln \frac{1+x}{c+x}\right)}\:\frac{1}{1-c}\right)^2dxdy}}\\
&=\int_c^1{t^2\left( \frac{1}{y+c}-\frac{1}{y+1}\right) \left( \frac{1}{\frac{1+p}{2p}-\left(1-\frac{x}{1-c}\ln \frac{1+x}{c+x}\right)}\:\frac{1}{1-c}\right)^2dy}.
\end{aligned} \]
The expression increases in $p$ since $\frac{1+p}{2p}$ in the denominator decreases. This implies that once we have found a pair $(p,c)$ for which $\| \widetilde{A}\|_{HS}^2<1$  then this inequality also holds for all pairs $(q,c)$ where $0<q\le p$. The region of pair of parameters for which $\| \widetilde{A}\|_{HS}^2<1$ can be determined numerically, and it is represented on Fig. \ref{409}. In addition, Fig. \ref{410} and \ref{411} illustrates the density $r_1(x)$ and  the distribution of $\tau$ on a log-log plot for three different choice of parameters, respectively.

  \begin{figure}
    \includegraphics[width=\textwidth]{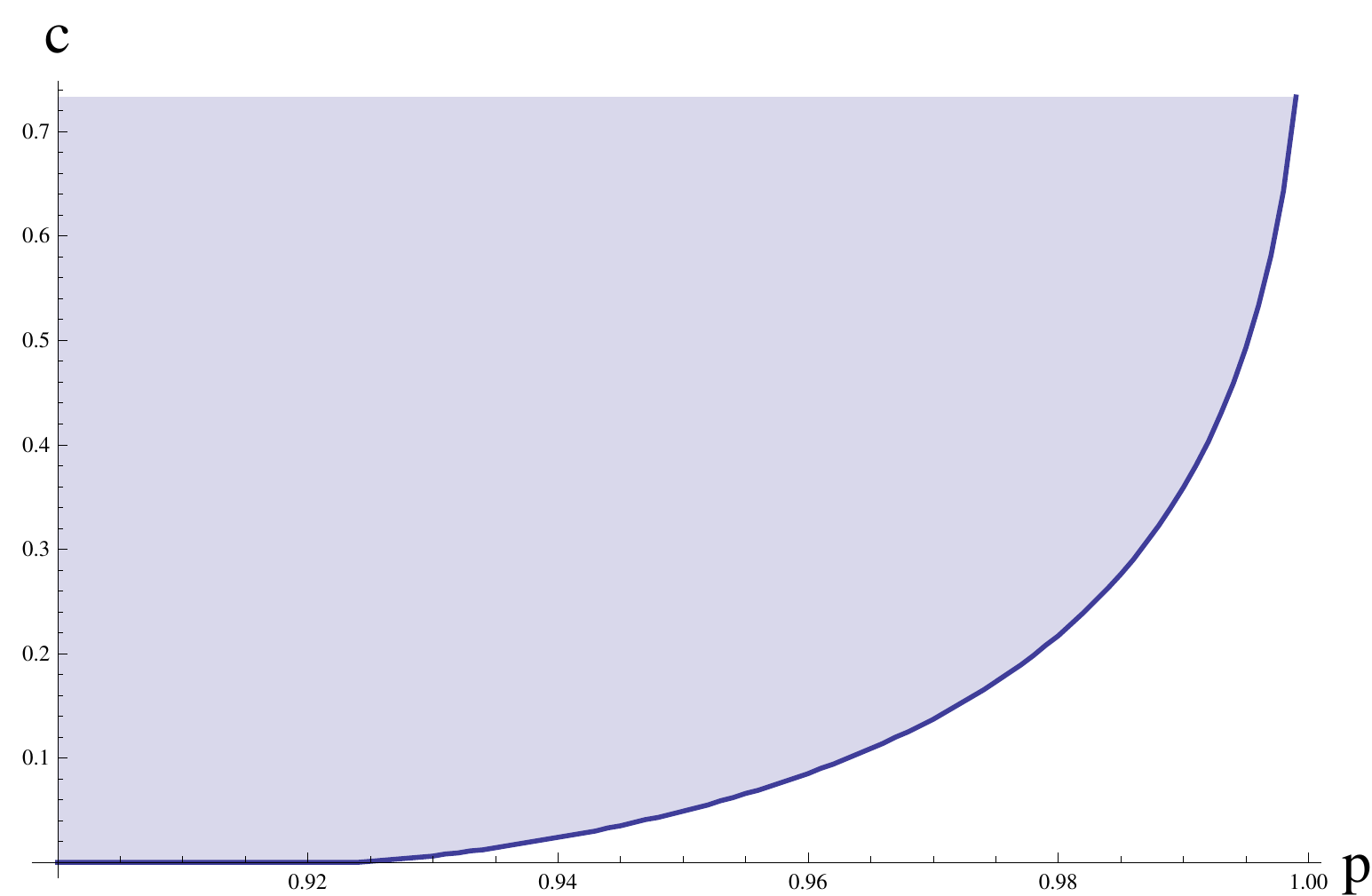}
    \caption{The set of parameters $(p,c)$ where we can guarantee convergence in \eqref{408}.}\label{409}
  \end{figure}

  \begin{figure}
    \includegraphics[width=\textwidth]{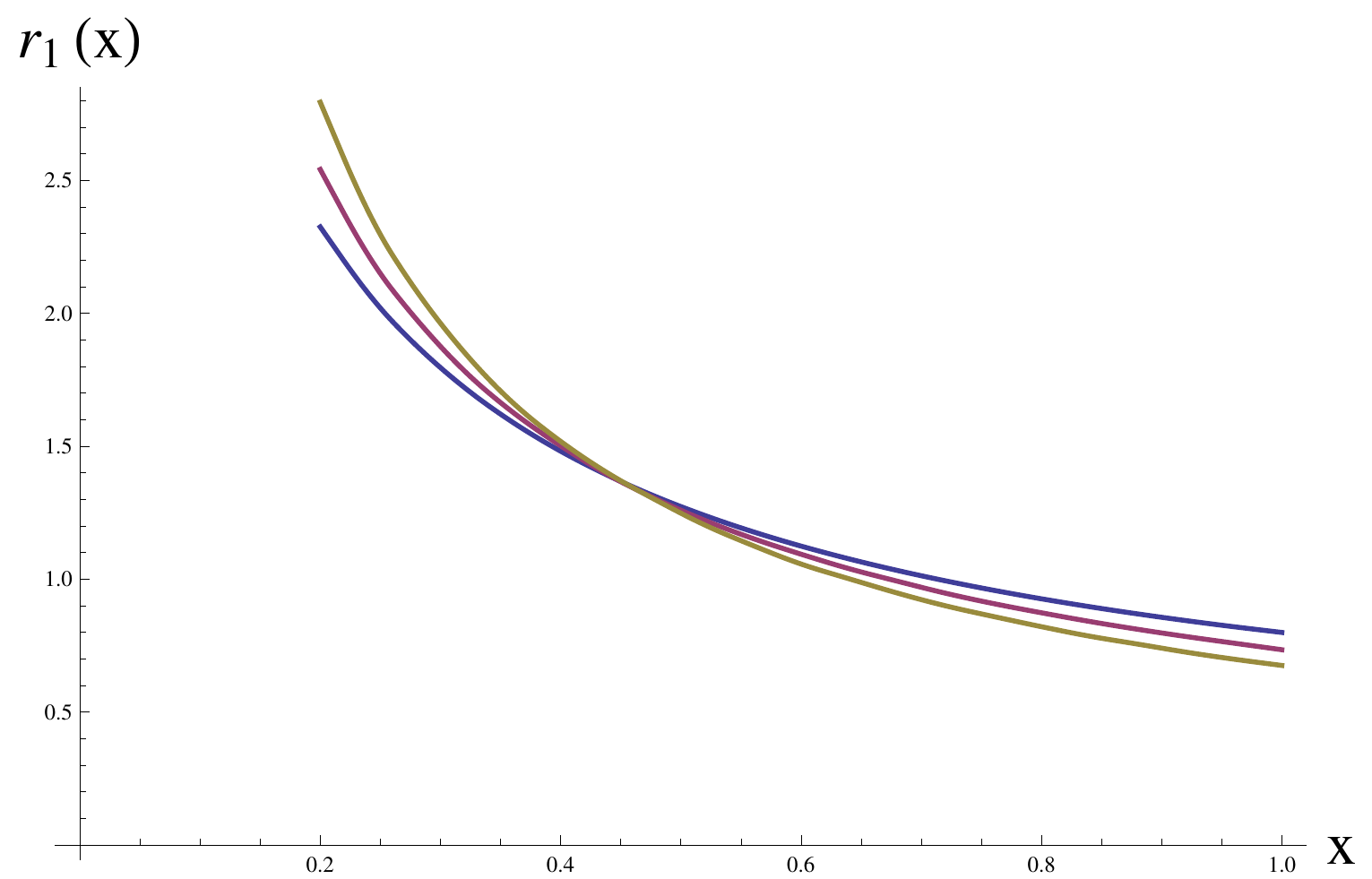}
    \caption{The DF of the old task's priority in the stationary case with $c=0.2$ and different values of $p$. Blue: $p=0.7$. Red: $p=0.8$. Mustard: $p=0.9$. }\label{410}
  \end{figure}

  \begin{figure}
    \includegraphics[width=\textwidth]{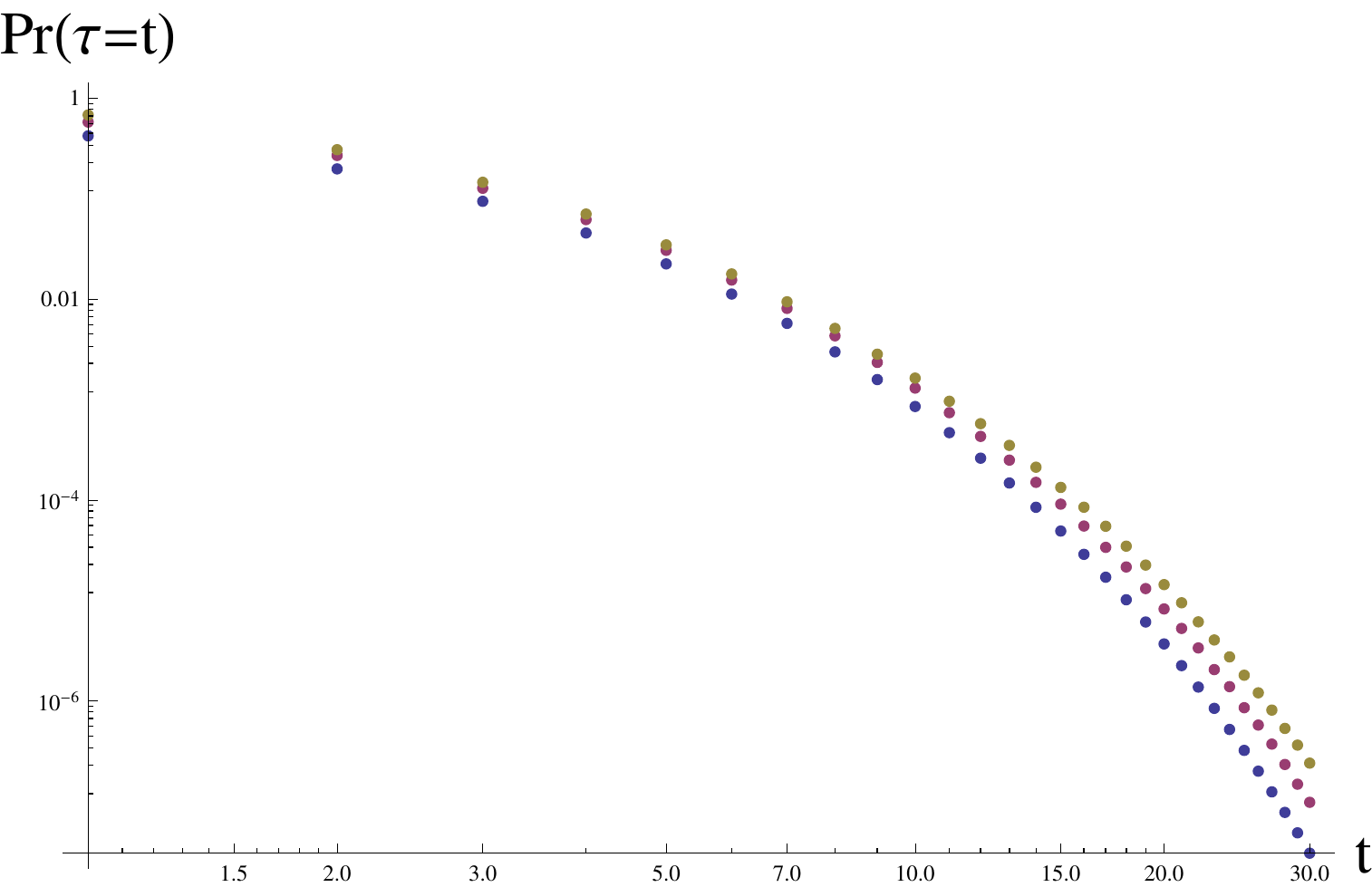}
    \caption{The distribution of $\tau$, the waiting time of a new task in the stationary case, with $c=0.2$ and different values of $p$ drawn on a log-log plot. Blue: $p=0.7$. Red: $p=0.8$. Mustard: $p=0.9$.}\label{411}
  \end{figure}

Now we give estimates on the distribution of $\tau$ when $k>1$, which is given in equation \eqref{310}.
In this example the functions $q(x)$ and $q_1(x)$ (defined in \eqref{507} and \eqref{5071}) can be written in the following form:
\[ \ba
&q(x)=\frac{p}{1-c}\int_{c}^{1}{\frac{y}{x+y}dy}+\frac{1-p}{2}=\frac{1+p}{2}-\frac{px}{1-c}\:\ln \frac{1+x}{c+x}\\
&q_{1}(x)=p\int_{c}^{1}{\frac{y}{x+y}dR_{1}(y)}+\frac{1-p}{2}.
\ea \]
It is easy to see that $q'(x)\neq 0$, so we can write
\[ \ba
\mathbb{P}(\tau=k)&=\int_{c}^{1}{q_{1}(x)(1-q(x))q(x)^{k-2}dR(x)}\\
&=\int_{c}^{1}{\underbrace{\frac{1}{1-c}\:\frac{q_{1}(x)(1-q(x))}{-q'(x)}}_{:=L(x)}(-q'(x))q(x)^{k-2}dx}.
\ea \]
With the new notation $L(x):=-\frac{1}{1-c}\:\frac{q_{1}(x)(1-q(x))}{q'(x)}$, we claim that  $L(x)$ can be bounded from below and from above by some constants $m, M$ only depending only on $(p,c)$:
\be \label{eq::lx_bounds}
m_{p,c}\leq L(x)\leq M_{p,c},\ee
since 
\[\begin{aligned}
\frac{1-p}{2}+pc\leq q_1(1) & \leq q_1(x)\leq 1, \\
\frac{1-p}{2}+pc\ln \frac{1}{2c} \leq 1-q(c) & \leq 1-q(x)\leq 1, \\
\frac{1-c}{\ln \frac{1+c}{2c}} \leq -\frac{1}{q'(c)} & \leq -\frac{1}{q'(x)}\leq -\frac{1}{q'(1)}\xrightarrow{(p,c)\to(1,0)} \ln 2 +\frac12.
\end{aligned}\]
Thus if we set
\[
m_{p,c}:=-\frac{q_1(1)(1-q(c))}{(1-c)q'(1)}\text{, and }
M{p,c}:=-\frac{1}{(1-c)q'(1)}.
\] then the bound \eqref{eq::lx_bounds} follows.
With these bounds, one can easily estimate $\mathbb{P}(\tau=k)$:
\[
\mathbb{P}(\tau=k)\leq M(p,c)\int_{c}^{1}{q'(x)q(x)^{k-2}dR(x)}=M(p,c)\frac{1}{k-1}(q^{k-1}(c)-q^{k-1}(1))\]
and similarly \[ \mathbb{P}(\tau=k)\ge m(p,c)\frac{1}{k-1}(q^{k-1}(c)-q^{k-1}(1)).\]
In the limit $(p,c)\to (1,0)$ we get
\[
\lim_{(p,c)\to (1,0)}\mathbb{P}(\tau=k)= \Omega \left( \frac{\frac{1-p}{2}+c\ln \frac{1}{2c}}{\ln \frac{1}{2c}}\left( \frac{1-p}{2}+c\right)\right)(1-(1-\ln 2)^{k-1})\frac{1}{k-1}
\]
Since $q(c)>q(1)$, hence
\be \label{412}
\mathbb{P}(\tau=k)\sim const\cdot\frac1kq^k(c)=const\cdot\frac1ke^{-k/k_0},
\ee
where
\[
k_0=-\frac{1}{\ln \left( q(c)\right) }.
\]
This means that the probability that the waiting time is of length $k$ behaves approximately as $\frac1k$ until $k<k_0$, i.e. it obeys a power law-like behavior. Furthermore, if $c\neq 0$ and $(p,c)\to (1,0)$, then $k_0\to \infty$, i.e. the exponential cutoff in equation \eqref{412} shifts to infinity, and thus for $(p,c)$ values close to $(1,0)$ the distribution of $\tau$ will be close to a power-law distribution.

\section{Summary}
In this paper we investigated and generalized the priority queueing model of Barabási. We showed that in the original model of Barabási, the system is ergodic and irreducible as a Markov chain once the priority selection probability is  separated away from $1$. Further, we gave a more probabilistic approach to compute the distribution of the priority of the old task in the system if the buffer length equals $2$, and determined the average waiting time of a task in the system for arbitrary buffer length. We investigated the case $p=1$ separately and found that the model is equivalent to the processes of records. Next, we generalized the priority queueing model with an arbitrary selection protocol depending on the priorities in the system and some extra randomness. We found that the system is ergodic and irreducible once the selection protocol is separated away from $1$. Further, we gave a description of the density of priority of the old task in terms of Hilbert-Schmidt operators. We investigated a special example when the priority selection is done proportionally to the priorities of the tasks in more detail.  Namely, using bounds and approximation methods based on the operator approach, we gave a region of pairs $(p,c)$ where the Hilbert-Schmidt description of the density function is converging and determined the waiting time distribution in this case.

\bibliographystyle{plain}

\end{document}